\documentclass[11pt]{amsart}
\usepackage{amssymb}

\newtheorem{theorem}{Theorem}[section]
\newtheorem{corollary}[theorem]{Corollary}
\newtheorem{lemma}[theorem]{Lemma}

\theoremstyle{definition}
\newtheorem{definition}[theorem]{Definition}

\theoremstyle{remark}

\numberwithin{equation}{section}

\begin{document}

\title{Characterization of $\ell_p$-like and $c_0$-like equivalence relations}
\author{Longyun Ding}
\address{School of Mathematical Sciences and LPMC, Nankai University, Tianjin, 300071, P.R.China}
\email{dinglongyun@gmail.com}
\thanks{Research partially supported by the National Natural Science Foundation of China (Grant No. 10701044).}


\date{\today}

\begin{abstract}
Let $X$ be a Polish space, $d$ a pseudo-metric on $X$. If $\{(u,v):d(u,v)<\delta\}$ is ${\bf\Pi}^1_1$ for each $\delta>0$,
we show that either $(X,d)$ is separable or there are $\delta>0$ and a perfect set $C\subseteq X$
such that $d(u,v)\ge\delta$ for distinct $u,v\in C$.

Granting this dichotomy, we characterize the positions of $\ell_p$-like and $c_0$-like equivalence relations in the Borel reducibility hierarchy.
\end{abstract}
\maketitle

\section{Introduction}

Let $(X,d)$ be a pseudo-metric space. If $X$ is not separable, by Zorn's lemma, we can easily prove that, there are $\delta>0$ and a noncountable set $C\subseteq X$
such that $d(u,v)\ge\delta$ for distinct $u,v\in C$. However, if we do not assume CH, can we find such a $C$ whose cardinal is of $2^\Bbb N$? J. H. Silver \cite{silver}
answered a similar problem for equivalence relations under an extra assumption of coanalyticity.

A topological space is called a {\it Polish space} if it is separable and completely metrizable. As usual, We denote the Borel, analytic and coanalytic sets by
${\bf\Delta}^1_1,{\bf\Sigma}^1_1$ and ${\bf\Pi}^1_1$ respectively. For their effective analogues, the Kleene pointclasses and
the relativized Kleene pointclasses are denoted by $\Delta^1_1,\Sigma^1_1,\Pi^1_1,\Delta^1_1(\alpha),\Sigma^1_1(\alpha),\Pi^1_1(\alpha)$, etc.
For more details in descriptive set theory, one can see \cite{kechris} and \cite{MK}.

\begin{theorem}[Silver]
Let $E$ be a ${\bf\Pi}^1_1$ equivalence relation on a Polish space. Then $E$ has
either at most countably many or perfectly many equivalence classes.
\end{theorem}

In section 2, we use the Gandy-Harrington topology to establish the following dichotomy.

\begin{theorem}
Let $X$ be a Polish space, $d$ a pseudo-metric on $X$. If $\{(u,v):d(u,v)<\delta\}$ is ${\bf\Pi}^1_1$ for each $\delta>0$,
then either $(X,d)$ is separable or there are $\delta>0$ and a perfect set $C\subseteq X$
such that $d(u,v)\ge\delta$ for distinct $u,v\in C$.
\end{theorem}

Let $X,Y$ be Polish spaces and
$E,F$ equivalence relations on $X,Y$ respectively. A {\it Borel
reduction} of $E$ to $F$  is a Borel function $\theta:X\to Y$ such
that $(x,y)\in E$ iff $(\theta(x),\theta(y))\in F$, for all $x,y\in
X$. We say that $E$ is {\it Borel reducible} to $F$, denoted $E\le_B
F$, if there is a Borel reduction of $E$ to $F$. If $E\le_B F$ and
$F\le_B E$, we say that $E$ and $F$ are {\it Borel bireducible} and
denote $E\sim_B F$. We refer to \cite{gao} for
background on Borel reducibility.

In section 3, we will introduce notions of $\ell_p$-like and $c_0$-like equivalence relations.
Granting the dichotomy on pseudo-metric spaces, we answer that when is $E_1$ Borel reducible to an $\ell_p$-like or a $c_0$-like equivalence relation.
In the end, we compare $\ell_p$-like and $c_0$-like equivalence relations with some remarkable equivalence
relations $E_0,E_1,E_0^\omega$.
\begin{enumerate}
\item[(a)] For $x,y\in 2^\Bbb N$, $(x,y)\in E_0\Leftrightarrow\exists m\forall
n\ge m(x(n)=y(n)).$
\item[(b)] For $x,y\in 2^{\Bbb N\times\Bbb N}$, $(x,y)\in
E_1\Leftrightarrow\exists m\forall n\ge m\forall k(x(n,k)=y(n,k)).$
\item[(c)] For $x,y\in 2^{\Bbb N\times\Bbb N}$, $(x,y)\in
E_0^\omega\Leftrightarrow\forall k\exists m\forall n\ge
m(x(n,k)=y(n,k)).$
\end{enumerate}

The following dichotomies show us why these equivalence relations are so remarkable.

\begin{theorem}
Let $E$ be a Borel equivalence relation. Then
\begin{enumerate}
\item[(a)] {\rm (Harrington-Kechris-Louveau \cite{HKL})} either $E\le_B{\rm
id}(\Bbb R)$ or $E_0\le_B E$;
\item[(b)] {\rm (Kechris-Louveau \cite{KL})} if $E\le_B E_1$, then
$E\le_B E_0$ or $E\sim_B E_1$;
\item[(c)] {\rm (Hjorth-Kechris \cite{HK})} if $E\le_B E_0^\omega$,
then $E\le_B E_0$ or $E\sim_B E_0^\omega$.
\end{enumerate}
\end{theorem}

\section{Separable or not}

For a ${\bf\Pi}^1_1$ equivalence relation $E$ on $X$, let us consider the following
pseudo-metric on $X$:
$$d_E(u,v)=\left\{\begin{array}{ll}0, & (u,v)\in E,\cr 1, &(u,v)\notin E.\end{array}\right.$$
From Silver's theorem, we can see that either $d_E$ is separable or there is a perfect set $C\subseteq X$ such that $d_E(u,v)=1$ for distinct $u,v\in C$.

By the same spirit of the Silver dichotomy theorem, we define:

\begin{definition}
Let $X$ be a Polish space, $d$ a pseudo-metric on $X$. If $\{(u,v):d(u,v)<\delta\}$ is ${\bf\Pi}^1_1$ for each $\delta>0$, we say $d$ is {\it lower} ${\bf\Pi}^1_1$.
\end{definition}

For a pseudo-metric space $(X,d)$ and $\delta>0$, we say $(X,d)$ is {\it $\delta$-separable} if there is a countable set $S\subseteq X$ such that
$$\forall u\in X\exists s\in S(d(u,s)<\delta).$$
Hence $(X,d)$ is separable iff it is $\delta$-separable for arbitrary $\delta>0$.

\begin{theorem} \label{separable}
Let $X$ be a Polish space, $d$ a lower ${\bf\Pi}^1_1$ pseudo-metric. Then for $\delta>0$, either $(X,d)$ is $\delta$-separable or there is a perfect set $C\subseteq X$
such that $d(u,v)\ge\delta/2$ for distinct $u,v\in C$.
\end{theorem}

\begin{proof}
We denote $Q=\{(u,v):d(u,v)<\delta\}$ and $R=\{(u,v):d(u,v)<\delta/2\}$. We see that both $Q,R$ are ${\bf\Pi}^1_1$.

Then the theorem follows from the next lemma.
\end{proof}

\begin{lemma}
Let $X$ be a Polish space, $Q,R\subseteq X^2$. Assume that
\begin{enumerate}
\item[(i)] $Q$ is ${\bf\Pi}^1_1$ and $R$ is $\sigma({\bf\Sigma}^1_1)$ (the $\sigma$-algebra generated by the ${\bf\Sigma}^1_1$ sets);
\item[(ii)] $\Delta(X)=\{(u,u):u\in X\}$ contains in $Q$;
\item[(iii)] if there exists $v\in X$ such that $(v,u)\in R,(v,w)\in R$, then $(u,w)\in Q$.
\end{enumerate}
Then one of the following holds:
\begin{enumerate}
\item[(a)] there is a countable set $S\subseteq X$ such that $\forall u\in X\exists s\in S((u,s)\in Q)$;
\item[(b)] there is a perfect set $C\subseteq X$ such that $(u,v)\notin R$ for distinct $u,v\in C$.
\end{enumerate}
\end{lemma}

\begin{proof}
We follow the method as in Harrington's proof for Silver's theorem.

Without loss of generality we may assume $X=\Bbb N^\Bbb N$ and $Q\in\Pi^1_1$. The proof for $Q\in\Pi^1_1(\alpha)$ with $\alpha\in\Bbb N^\Bbb N$ is similar.
Let $\tau$ be the Gandy-Harrington topology (the topology generated by all $\Sigma^1_1$ sets) on $\Bbb N^\Bbb N$.

For $u\in X$ we denote $Q(u)=\{v\in X:(u,v)\in Q\}$. First we define
$$V=\{u\in X:\mbox{ there is no $\Delta^1_1$ set $U$ such that }u\in U\subseteq Q(u)\}.$$

If $V=\emptyset$, since there are only countably many $\Delta^1_1$ set, we can find a countable subset $S\subseteq X$ which meets every nonempty $\Delta^1_1$ set
at least one point. For each $u\in X$ there is a nonempty $\Delta^1_1$ set $U\subseteq Q(u)$. Let $s\in S\cap U$. Then $s\in Q(u)$, i.e. $(u,s)\in Q$.

For the rest of the proof we assume $V\ne\emptyset$. Note that

$$u\in V\iff\forall U\in\Delta^1_1(u\in U\to\exists v\in U(v\notin Q(u))).$$
With the coding of $\Delta^1_1$ sets (see \cite{gao} Theorem 1.7.4), there are $\Pi^1_1$ subsets $P^+,P^-\subseteq\Bbb N\times\Bbb N^\Bbb N$ and $D\subseteq\Bbb N$ such that
\begin{enumerate}
\item $\forall n\in D\forall u((n,u)\in P^+\Leftrightarrow(n,u)\notin P^-)$;
\item for any $\Delta^1_1$ set $A$ there is $n\in D$ such that $\forall u(u\in A\Leftrightarrow(n,u)\in P^+)$.
\end{enumerate}
Thus we have
$$u\in V\iff\forall n((n\in D,(n,u)\in P^+)\to\exists v((n,v)\notin P^-,(u,v)\notin Q)).$$
So $V$ is $\Sigma^1_1$.

By a theorem of Nikodym (see \cite{kechris} Corollary 29.14), the class of sets with the Baire property in any topological space is closed under the
Suslin operation. It is well known that all ${\bf\Sigma}^1_1$ sets are results of the Suslin operation applied on closed sets in the usual topology
(see \cite{kechris} Theorem 25.7). Note that all closed sets in usual topology are also closed in $\tau$, we see that every $\sigma({\bf\Sigma}^1_1)$
subset of $\Bbb N^\Bbb N$ (or $\Bbb N^\Bbb N\times\Bbb N^\Bbb N$) has Baire property in $\tau$ (or $\tau\times\tau$).

Toward a contradiction assume that for some $v\in V$, $R(v)$ is not $\tau$-meager in $V$. Since $R(v)$ has Baire property in $\tau$,
there is a nonempty $\Sigma^1_1$ set $U\subseteq V$
such that $R(v)$ is $\tau$-comeager in $U$. By Louveau's lemma (see \cite{MK} Lemma 9.3.2), $R(v)\times R(v)$ meets any nonempty $\Sigma^1_1$ set in $U\times U$.
We denote $\neg Q=(\Bbb N^\Bbb N\times\Bbb N^\Bbb N)\setminus Q$. If $\neg Q\cap(U\times U)\ne\emptyset$, since it is $\Sigma^1_1$, we have
$$(R(v)\times R(v))\cap\neg Q\cap(U\times U)\ne\emptyset,$$
which contradicts to clause (iii). Thus we have $U\times U\subseteq Q$. We define $W$ by
$$w\in W\iff\forall u(u\in U\to(u,w)\in Q).$$
Fix a $u_0\in U$. We can see that $W$ is $\Pi^1_1$ and $U\subseteq W\subseteq Q(u_0)$. By the separation property for $\Sigma^1_1$ sets there is $U_0\in\Delta^1_1$
such that $U\subseteq U_0\subseteq W$. Then we have $u_0\in U_0\subseteq Q(u_0)$, which contradicts $u_0\in U\subseteq V$. Therefore, $R(v)$ is $\tau$-meager in $V$ for
each $v\in V$.

Since $R$ has Baire property in $\tau\times\tau$, by the Kuratowski-Ulam theorem (see \cite{kechris} Theorem 8.41), $R$ is $\tau\times\tau$-meager in $V\times V$.
By the definition of $V$ and clause (ii),
we see that $V$ contains no $\Delta^1_1$ real, i.e. $V$ has no isolate point in $\tau$. Since the space $\Bbb N^\Bbb N$ with $\tau$ is strong Choquet
(see \cite{gao} Theorem 4.1.5), $V$ is a perfect Choquet space. From \cite{kechris} Exercise 19.5, we can find a perfect set $C\subseteq V$ such that
$(u,v)\notin R$ for distinct $u,v\in C$.
\end{proof}

\section{Characterization}

The notion of $\ell_p$-like equivalence relation was introduced in \cite{ding1}.

\begin{definition}
Let $(X_n,d_n),\,n\in\Bbb N$ be a sequence of pseudo-metric spaces, $p\ge 1$. We define an equivalence relation
$E((X_n,d_n)_{n\in\Bbb N};p)$ on $\prod_{n\in\Bbb N}X_n$ by
$$(x,y)\in E((X_n,d_n)_{n\in\Bbb N};p)\iff\sum_{n\in\Bbb
N}d_n(x(n),y(n))^p<+\infty$$ for $x,y\in\prod_{n\in\Bbb N}X_n$.
We call it an {\it $\ell_p$-like equivalence relation}. If $(X_n,d_n)=(X,d)$ for every $n\in\Bbb N$, we write
$E((X,d);p)=E((X_n,d_n)_{n\in\Bbb N};p)$ for the sake of brevity.
\end{definition}

If $X$ is a separable Banach space, we have $E(X;p)=X^\Bbb N/\ell_p(X)$ where $\ell_p(X)$ is the Banach space whose underlying space is
$\{x\in X^\Bbb N:\sum_{n\in\Bbb N}\|x(n)\|^p<+\infty\}$ with the norm $\|x\|=\left(\sum_{n\in\Bbb N}\|x(n)\|^p\right)^{\frac{1}{p}}$.
Then $E(X;p)$ is an orbit equivalence relation induced by a Polish group action, thus $E_1\not\le_B E(X;p)$ (see \cite{gao} Theorem 10.6.1).

Let $(X,d)$ be a pseudo-metric space, we denote
$$\delta(X)=\inf\{\delta:X\mbox{ is $\delta$-separable}\}.$$

\begin{theorem} \label{lp}
Let $X_n,\,n\in\Bbb N$ be a sequence of Polish spaces, $d_n$ a Borel pseudo-metric on $X_n$ for each $n$ and $p\ge 1$. Denote $E=E((X_n,d_n)_{n\in\Bbb N};p)$.
We have
\begin{enumerate}
\item[(i)] $\sum_{n\in\Bbb N}\delta(X_n)^p<+\infty\iff E\le_B E(c_0;p)$;
\item[(ii)] $\sum_{n\in\Bbb N}\delta(X_n)^p=+\infty\iff E_1\le_B E$.
\end{enumerate}
\end{theorem}

\begin{proof}
Because $E_1\not\le_B E(c_0;p)$, we only need to prove ($\Rightarrow$) for (i) and (ii).

(i) By the definition of $\delta(X_n)$, we see that $X_n$ is $(\delta(X_n)+2^{-n})$-separable, i.e. there is a countable set $S_n\subseteq X_n$ such that
$$\forall u\in X_n\exists s\in S_n(d_n(u,s)<\delta(X_n)+2^{-n}).$$
Let $S_n=\{s_m^n:m\in\Bbb N\}$. Without loss of generality, we assume that $d_n(s_k^n,s_l^n)>0$ for $k\ne l$, i.e. $d_n$ is a metric on $S_n$.
For $u\in X$ we denote $m(u)$ the least $m$ such that $d(u,s_m^n)<\delta(X_n)+2^{-n}$. Then we define $h_n:X_n\to S_n$ by
$h_n(u)=s_{m(u)}^n$ for $u\in X$. It is easy to see that $h_n$ is Borel. Define $\theta:\prod_{n\in\Bbb N}X_n\to\prod_{n\in\Bbb N}S_n$ by
$$\theta(x)(n)=h_n(x(n))$$
for $x\in\prod_{n\in\Bbb N}X_n$. Note that for each $x$ we have
$$\sum_{n\in\Bbb N}d_n(x(n),\theta(x)(n))^p<\sum_{n\in\Bbb N}(\delta(X_n)+2^{-n})^p\le2^{p-1}\sum_{n\in\Bbb N}(\delta(X_n)^p+2^{-np})<+\infty,$$
i.e. $(x,\theta(x))\in E$. It follows that $(x,y)\in E\Leftrightarrow(\theta(x),\theta(y))\in E$. Hence $E\le_B E((S_n,d_n)_{n\in\Bbb N};p)$.

Note that each $(S_n,d_n)$ is a separable metric space. From Aharoni's theorem \cite{aharoni}, there are $K>0$ and $T_n:S_n\to c_0$ satisfying
$$d_n(u,v)\le\|T_n(u)-T_n(v)\|_{c_0}\le Kd_n(u,v)$$
for every $u,v\in S_n$. Define $\theta_1:\prod_{n\in\Bbb N}S_n\to c_0^\Bbb N$ by
$$\theta_1(x)(n)=T_n(x(n))$$
for $x\in\prod_{n\in\Bbb N}S_n$. It is easy to check that $\theta_1$ is a Borel reduction of $E((S_n,d_n)_{n\in\Bbb N};p)$ to $E(c_0;p)$.

(ii) Without loss of generality, we may assume that $\delta(X_n)>0$ for each $n$. Select a sequence $0<\delta_n<\delta(X_n),\,n\in\Bbb N$ such that
$\sum_{n\in\Bbb N}\delta_n^p=+\infty$. Thus we can find a strictly increasing sequence of natural numbers $(n_j)_{j\in\Bbb N}$ such that $n_0=0$ and
$$\sum_{n=n_j}^{n_{j+1}-1}\delta_n^p\ge 2^p,\quad j=0,1,2,\cdots.$$

Since $(X_n,d_n)$ is not $\delta_n$-separable, from Theorem \ref{separable}, there is a Borel injection
$g_n:2^\Bbb N\to X_n$ such that $d_n(g_n(\alpha),g_n(\beta))\ge\delta_n/2$ for distinct $\alpha,\beta\in 2^\Bbb N$.
Define $\vartheta:(2^\Bbb N)^\Bbb N\to\prod_{n\in\Bbb N}X_n$ by
$$\vartheta(x)(n)=g_n(x(j)),\quad n_j\le n<n_{j+1}.$$
For each $x,y\in(2^\Bbb N)^\Bbb N$, if $x(j)\ne y(j)$ for some $j\in\Bbb N$, we have
$$\sum_{n=n_j}^{n_{j+1}-1}\!\!\!d_n(\vartheta(x)(n),\vartheta(y)(n))^p=\!\!\!\sum_{n=n_j}^{n_{j+1}-1}\!\!d_n(g_n(x(j)),g_n(y(j)))^p
\ge\!\!\!\sum_{n=n_j}^{n_{j+1}-1}\!\!\!(\delta_n/2)^p\ge 1.$$
Therefore
$$\begin{array}{ll}&(\vartheta(x),\vartheta(y))\in E\cr
\iff &\sum_{j\in\Bbb N}\sum_{n=n_j}^{n_{j+1}-1}d_n(\vartheta(x),\vartheta(y))^p<+\infty\cr
\iff &\exists k\forall j>k(x(j)=y(j))\cr
\iff &(x,y)\in E_1.\end{array}$$
Thus $\vartheta$ witnesses that $E_1\le_B E$.
\end{proof}

$c_0$-like equivalence relations were first studied by I. Farah \cite{farah}.

\begin{definition}
Let $(X_n,d_n),\,n\in\Bbb N$ be a sequence of pseudo-metric spaces. We define an equivalence relation
$E((X_n,d_n)_{n\in\Bbb N};0)$ on $\prod_{n\in\Bbb N}X_n$ by
$$(x,y)\in E((X_n,d_n)_{n\in\Bbb N};0)\iff\lim_{n\to\infty}d_n(x(n),y(n))=0$$ for $x,y\in\prod_{n\in\Bbb N}X_n$.
We call it a {\it $c_0$-like equivalence relation}. If $(X_n,d_n)=(X,d)$ for every $n\in\Bbb N$, we write
$E((X,d);0)=E((X_n,d_n)_{n\in\Bbb N};0)$ for the sake of brevity.
\end{definition}

Farah mainly investigated the case named $c_0$-equalities that all $(X_n,d_n)$'s are finite metric spaces and denoted it by ${\rm D}(\langle X_n,d_n\rangle)$.

\begin{theorem} \label{c0}
Let $X_n,\,n\in\Bbb N$ be a sequence of Polish spaces, $d_n$ a Borel pseudo-metric on $X_n$ for each $n$. Denote $E=E((X_n,d_n)_{n\in\Bbb N};0)$.
We have
\begin{enumerate}
\item[(i)] $\lim_{n\to\infty}\delta(X_n)=0\iff E\le_B\Bbb R^\Bbb N/c_0$;
\item[(ii)] $(\delta(X_n))_{n\in\Bbb N}$ does not converge to $0\iff E_1\le_B E$.
\end{enumerate}
\end{theorem}

\begin{proof}
We closely follows the proof of Theorem 3.2. Some conclusions will be made without proofs for brevity, since they follow by similar arguments.

(i) Note that for each $x$ we have
$$\lim_{n\to\infty}d_n(x(n),\theta(x)(n))=\lim_{n\to\infty}(\delta(X_n)+2^{-n})=0,$$
i.e. $(x,\theta(x))\in E$. It follows that $(x,y)\in E\Leftrightarrow(\theta(x),\theta(y))\in E$. Hence $E\le_B E((S_n,d_n)_{n\in\Bbb N};0)$.

Fix a bijection $\langle\cdot,\cdot\rangle:\Bbb N^2\to\Bbb N$. We define $\theta_2:\prod_{n\in\Bbb N}S_n\to\Bbb R^\Bbb N$ by
$$\theta_2(x)(\langle n,m\rangle)=T_n(x(n))(m)$$
for $x\in\prod_{n\in\Bbb N}S_n$ and $n,m\in\Bbb N$. It is easy to
see that $\theta_2$ is Borel.

Now we check that $\theta_2$ is a reduction.

For every $x,y\in\prod_{n\in\Bbb N}S_n$, if $(x,y)\in
E((S_n,d_n)_{n\in\Bbb N};0)$, then
$$\lim_{n\to\infty}d_n(x(n),y(n))\to 0.$$
So $\forall\varepsilon>0\exists N\forall
n>N(d_n(x(n),y(n))<\varepsilon)$. Since
$\|T_n(x(n))-T_n(y(n))\|_{c_0}\le Kd_n(x(n),y(n))<K\varepsilon$, we
have
$$\forall n>N\forall m(|T_n(x(n))(m)-T_n(y(n))(m)|<K\varepsilon).$$
For $n\le N$, since $T_n(x(n)),T_n(y(n))\in c_0$, we have
$$\lim_{m\to\infty}|T_n(x(n))(m)-T_n(y(n))(m)|=0.$$
Therefore, for all
but finitely many $(n,m)$'s, we have
$$|\theta_2(x)(\langle n,m\rangle)-\theta_2(y)(\langle
n,m\rangle)|=|T_n(x(n))(m)-T_n(y(n))(m)|<K\varepsilon.$$ Thus
$$\lim_{\langle n,m\rangle\to\infty}|\theta_2(x)(\langle n,m\rangle)-\theta_2(y)(\langle
n,m\rangle)|=0.$$ It follows that $\theta_2(x)-\theta_2(y)\in c_0$.

On the other hand, for every $x,y\in\prod_{n\in\Bbb N}S_n$, if
$\theta_2(x)-\theta_2(y)\in c_0$, then
$$\forall\varepsilon>0\exists N\forall n>N\forall m(|\theta_2(x)(\langle
n,m\rangle)-\theta_2(y)(\langle n,m\rangle)|<\varepsilon).$$
Therefore, for $n>N$ we have
$$\begin{array}{ll}d_n(x(n),y(n))&\le\|T_n(x(n))-T_n(y(n))\|_{c_0}\cr
&=\sup\limits_{m\in\Bbb N}|T_n(x(n))(m)-T_n(y(n))(m)|\cr
&=\sup\limits_{m\in\Bbb N}|\theta_2(x)(\langle
n,m\rangle)-\theta_2(y)(\langle
n,m\rangle)|\le\varepsilon.\end{array}$$ It follows that $(x,y)\in
E((S_n,d_n)_{n\in\Bbb N};0)$.

To sum up, we have $E\le_B E((S_n,d_n)_{n\in\Bbb N};0)\le_B\Bbb R^\Bbb N/c_0$.

(ii) Assume that $(\delta(X_n))_{n\in\Bbb N}$ does not converge to $0$. Then there are $c>0$ and a strictly increasing sequence of
natural numbers $(n_j)_{j\in\Bbb N}$ such that $\delta(X_{n_j})>c$ for each $j$. From Theorem \ref{separable}, there is a Borel injection
$g_j':2^\Bbb N\to X_{n_j}$ such that $d_{n_j}(g_j'(\alpha),g_j'(\beta))\ge c/2$ for distinct $\alpha,\beta\in 2^\Bbb N$.
Fix an element $a_n\in X_n$ for every $n\in\Bbb N$. Define $\vartheta':(2^\Bbb N)^\Bbb N\to\prod_{n\in\Bbb N}X_n$ by
$$\vartheta'(x)(n)=\left\{\begin{array}{ll}g_j'(x(j)), & n=n_j,\cr a_n, &\mbox{otherwise}.\end{array}\right.$$
Then $\vartheta'$ witnesses that $E_1\le_B E$.
\end{proof}

\section{Further remarks}

The following condition was introduced in \cite{ding1} to investigate the position of $\ell_p$-like equivalence relations.

$(\ell 1)$ {\it $\forall c>0\exists x,y\in\prod_{n\in\Bbb N}X_n$
such that $\forall n(d_n(x(n),y(n))^p<c)$ and
$$\sum_{n\in\Bbb N}d_n(x(n),y(n))^p=+\infty.$$}

Let $X_n,\,n\in\Bbb N$ be a sequence of Polish spaces, $d_n$ a Borel pseudo-metric on $X_n$ for each $n$ and $p\ge 1$. Denote $E=E((X_n,d_n)_{n\in\Bbb N};p)$.
It was proved in \cite{ding1} that
\begin{enumerate}
\item[(i)] if $(\ell 1)$ holds, then $\Bbb R^\Bbb N/\ell_1\le_B E$;
\item[(ii)] if $(\ell 1)$ fails, then either $E_1\le_B E,E\sim_B E_0$ or $E$ is trivial, i.e. all elements in $\prod_{n\in\Bbb N}X_n$ are equivalent.
\end{enumerate}

Thus we have a corollary of Theorem \ref{lp}.

\begin{corollary}
Denote $E=E((X_n,d_n)_{n\in\Bbb N};p)$. We have
\begin{enumerate}
\item[(a)] $\sum_{n\in\Bbb N}\delta(X_n)^p<+\infty$ and $(\ell 1)$ fails $\iff E\sim_B E_0$ or $E$ is trivial;
\item[(b)] $\sum_{n\in\Bbb N}\delta(X_n)^p<+\infty$ and $(\ell 1)$ holds $\iff\Bbb R^\Bbb N/\ell_1\le_B E\le_B E(c_0;p)$;
\item[(c)] $\sum_{n\in\Bbb N}\delta(X_n)^p=+\infty\iff E_1\le_B E$.
\end{enumerate}
\end{corollary}

Another condition was introduced by I. Farah \cite{farah} for investigating $c_0$-equalities.

$(*)$ {\it $\forall c>0\exists\varepsilon<c(\varepsilon>0$ and $\exists^\infty i\exists u_i,v_i\in X_i(\varepsilon<d_i(u_i,v_i)<c))$.}

It is easy to check that $(*)$ holds iff for arbitrary $c>0$, there exist $x,y\in\prod_{n\in\Bbb N}X_n$ such that $\forall n(d_n(x(n),y(n))<c)$ and
$(d_n(x(n),y(n)))_{n\in\Bbb N}$ does not converge to $0$.

With similar arguments, we get a corollary of Theorem \ref{c0}.

\begin{corollary}
Denote $E=E((X_n,d_n)_{n\in\Bbb N};0)$. We have
\begin{enumerate}
\item[(a)] $\lim_{n\to\infty}\delta(X_n)=0$ and $(*)$ fails $\iff E\sim_B E_0$ or $E$ is trivial;
\item[(b)] $\lim_{n\to\infty}\delta(X_n)=0$ and $(*)$ holds $\iff E_0^\omega\le_B E\le_B\Bbb R^\Bbb N/c_0$;
\item[(c)] $(\delta(X_n))_{n\in\Bbb N}$ does not converge to $0\iff E_1\le_B E$.
\end{enumerate}
\end{corollary}

\end{document}